\documentclass[11pt,reqno]{amsart}
\usepackage{amsmath, amssymb, amsfonts, xcolor}
\usepackage{hyperref}
\numberwithin{equation}{section}
\usepackage{amsthm}
\newtheorem{thm}{Theorem}[section]

\newtheorem{prop}[thm]{Proposition}
\newtheorem{lem}[thm]{Lemma}
\newtheorem{conj}[thm]{Conjecture}
\newtheorem{dfn}[thm]{Definition}

\newtheorem{remark}[thm]{Remark}

\usepackage[a4paper, top = 1.2in, bottom = 1.2in, left = 1.2in, right = 1.2in]{geometry}

\numberwithin{equation}{section}

\newcommand{\F}{\mathbb{F}}
\newcommand{\N}{\mathbb{N}}
\newcommand{\Q}{\mathbb{Q}}

\newcommand{\Z}{\mathbb{Z}}
\newcommand{\mbP}{\mathbb{P}}

\newcommand{\mcO}{\mathcal{O}}

\newcommand{\mfm}{\mathfrak{m}}
\newcommand{\mfn}{\mathfrak{n}}

\newcommand{\mfq}{\mathfrak{q}}
\newcommand{\mfP}{\mathfrak{P}}

\newcommand{\mfS}{\mathfrak{S}}

\newcommand{\GL}{\mathrm{GL}}

\newcommand{\Cl}{\mathrm{Cl}}

\newcommand{\Gal}{\mathrm{Gal}}

\def\1{1\!\!1}
\newcommand{\mrm}[1]{\mathrm{#1}}

\title[On the solutions of Diophantine equations]{On the solutions of $x^2= By^p+Cz^p$ and $2x^2= By^p+Cz^p$ over totally real fields}

\author[N. Kumar]{Narasimha Kumar}
\address[N. Kumar]{Department of Mathematics, Indian Institute of Technology Hyderabad, Kandi, Sangareddy 502285, INDIA.}
\email{narasimha@math.iith.ac.in}

\author[S. Sahoo]{Satyabrat Sahoo}
\address[S. Sahoo]{Department of Mathematics, Indian Institute of Technology Hyderabad, Kandi, Sangareddy 502285, INDIA.}
\email{ma18resch11004@iith.ac.in}

\date{\today}

\begin{document}
	\begin{abstract}
      In this article, we study the solutions of certain type over a totally real number field $K$ of the Diophantine equation $x^2= By^p+Cz^p$ with prime exponent $p$, where $B$ is an odd integer and $C$ is either an odd integer or $C=2^r$ for $r \in \N$. Further, we study the non-trivial primitive solutions of the Diophantine equation $x^2= By^p+2^rz^p$ ($r\in {1,2,4,5}$) (resp., $2x^2= By^p+2^rz^p$ with $r \in \N$) with prime exponent $p$, over $K$. We also present several purely local criteria of $K$.
	\end{abstract}
	\maketitle
     
     \section{Literature, Results \& Methodology: An Overview}
    Throughout this article, $K$, $\mathbb{P}$, and $p$, denote a totally real number field, the set of all rational primes, and a rational prime, respectively. The study of non-trivial solutions to Diophantine equations is one of the exciting and most interesting areas in mathematics. A prominent and most interesting example is the Fermat equation $x^p+y^p=z^p$ with exponent $p$.
    In~\cite[Theorem 0.5]{W95}, Wiles used the modularity of elliptic curves over rationals to show that $\Z$-integral
    solutions of $x^p+y^p=z^p$ are trivial. In~\cite[Theorem 1.3]{JM04}, Jarvis and Meekin showed 
    the same result that continues to hold even over $\Z[\sqrt{2}]$. A similar study over $K$ has been initiated by Freitas and Siksek in~\cite[Theorem 3]{FS15} and showed that an asymptotic version of these results continues to hold over $K$, by employing some explicit bounds on the solutions of $S$-unit equation. In~\cite[Theorem 1]{D16}, Deconinck generalized ~\cite[Theorem 3]{FS15} to $Ax^p+By^p=Cz^p$ with $2 \nmid ABC$. 

In~\cite[Theorem 3]{R97}, Ribet showed that $x^p+2^ry^p+z^p=0$ with exponent $p$ has no non-trivial 
    $\Z$-solution for $1\leq r <p$. Over $K$, in~\cite[Theorem 3.2]{KS22}, we show that the equation $x^p+y^p=2^rz^p\ (r \in \N)$ has no asymptotic solution in $W_K$ (cf. \cite[Definition 3.2]{KS22} for $W_K$). Furthermore, we show that the equation $x^p+y^p=2^rz^p$ has no asymptotic solution in $\mcO_K^3$ for $r=2,3$ (cf. \cite[Theorem 3.3]{KS22}). The proofs of our results also depend on certain explicit bounds on the solutions of the $S$-unit equation.

    Similarly, Ivorra in~\cite{I03} studied $\Z$-solutions of $x^2=y^p+2^rz^p$ and $2x^2=y^p+2^rz^p$ for $0 \leq r < p$. In ~\cite[Theorem 1]{S03}, Siksek established that the only non-trivial primitive $\Z$-solutions of $x^2=y^p+2^rz^p$ are $r=3,\ x=\pm3,\ y=z=1$, where $p \in \mathbb{P}$ is arbitrary.   

   Before proceeding, let us define the term ``asymptotic'' in the context of Diophantine equations, clarifying the conditions under which these equations are said to not have ``asymptotic solutions''.
 
     \begin{dfn}
    	We say a Diophantine equation $Ax^2= By^p+Cz^p$ with exponent $p$ has no asymptotic solution in a set $S \subseteq \mcO_K^3$, if there exists a constant $V_{K,A,B,C}>0$ (depending on $K,A,B,C$) such that for primes $p>V_{K,A,B,C}$, the equation $Ax^2= By^p+Cz^p$ has no non-trivial primitive solution in $S$.
    \end{dfn}

    In~\cite{DM97}, Darmon and Merel showed that the equation $x^n+y^n=z^2$ with exponent $n \geq 4$ has no non-trivial primitive $\Z$-solutions.  In~\cite[Theorem 1.1]{IKO20}, I\c{s}ik, Kara, and Ozman proved that the ``asymptotic'' FLT holds for $x^p+y^p=z^2$ with exponent $p$ of certain type over $K$, whenever the narrow class number $h_K^+=1$ and there exists a $\mathfrak{P} \in P:=\mrm{Spec}(\mcO_K)$ lying above $2$ with residual degree $1$. The proof of ~\cite[Theorem 1.1]{IKO20} relies on certain explicit bounds on the solutions of the $S$-unit equation. ~\cite[Theorem 1.1]{IKO20} was extended in two different directions as follows.
    \begin{itemize}
     \item In ~\cite[Theorem 5.2]{KS22}, we relaxed the assumptions in ~\cite[Theorem 1.1]{IKO20} and proved that the equation $x^p+y^p=z^2$ has no asymptotic solution in $W_K^\prime$ (cf. \cite[Definition 5.2]{KS22} for $W_K^\prime$). 
     
    \item  In ~\cite[Theorem 3]{M22}, Mocanu generalized \cite[Theorem 1.1]{IKO20} by replacing the assumption from $h_K^+=1$ to $\Cl_{S_K}(K)=1$ and its proof depends on some explicit bounds on the solutions of $\alpha +\beta =\gamma^2$, with $\alpha, \beta \in \mcO_{S_K}^\ast$ and $\gamma \in  \mcO_{S_K}$ (cf. \S \ref{section for main result of x^2=By^p+2^rz^p} for the definitions of $\Cl_{S_K}(K)$, $\mcO_{S_K}$ and $\mcO_{S_K}^\ast$).
     \end{itemize}
     In the first part, we extend~\cite[Theorem 3]{M22} to $x^2= By^p+Cz^p$ with exponent $p$ over $K$, where $B$ is an odd integer and $C$ is either an odd integer or $C=2^r$ for $r \in \N$. In Theorem~\ref{main result1 for x^2=By^p+Cz^p Type I}, we prove that the equation $x^2= By^p+Cz^p$ with exponent $p$ has no asymptotic solution in $W_K$ (cf. Definition~\ref{def for W_K} for $W_K$). For $r=1,2,4,5$, we also prove that $x^2= By^p+2^rz^p$ with exponent $p$ has no asymptotic solution in $\mcO_K^3 \setminus S_r$ for some finite set $S_r$ (cf. Theorem~\ref{main result1 for x^2=By^p+2^rz^p over K}). In Proposition~\ref{cor to main result2 for x^2=By^p+2^rz^p over K}, we provide situations where $S_r= \phi$. 
     
    In the second part, we perform a similar analysis to the equation $2x^2= By^p+2^rz^p  \ (r \in \N)$ with exponent $p$ and show that it has no asymptotic solution in $\mcO_K^3$ (cf. Theorem~\ref{main result1 for 2x^2=By^p+Cz^p Type I}).
 The proofs of these results rely on certain explicit bounds on the solutions of the equation
 \begin{equation}
 	\label{keyequation}
 	\alpha+\beta=\gamma^2,
 \end{equation}
 with $\alpha, \beta \in \mcO_{S_K^{\prime}}^\ast, \gamma \in \mcO_{S_K^{\prime}}$ (cf. \S \ref{notations section for x^2=By^p+Cz^p} for the definition of $S_K^{\prime}$).
 
  In the last part, we provide several local criteria of $K$ for Theorems~\ref{main result1 for x^2=By^p+Cz^p Type I}, ~\ref{main result1 for x^2=By^p+2^rz^p over K}, which follow from Propositions~\ref{main result2 for x^2=By^p+Cz^p Type I}, ~\ref{main result2 for  x^2=By^p+2^rz^p over K}, respectively, when $S_K^{\prime}=S_K$. Though these propositions can be thought of as an analog of ~\cite[Theorem 5]{M22}, they differ technically in the following two aspects:
    \begin{itemize}
    	\item The proof of~\cite[Theorem 5]{M22} uses the assumption ``$2$ is inert in $K$'', while we relax this assumption to ``$\mfP \in P$ is principal for all $\mfP|2$''.
    	\item In the final steps of the proof of ~\cite[Theorem 5]{M22}, the author needs to study the ramification behavior of a certain extension $L/K$. In our proofs, we need to invoke~\cite[Theorem 9(b)]{FKS20} in Proposition~\ref{main result2 for  x^2=By^p+2^rz^p over K} to study the extension $L/K$, where the discriminant of $L/K$ involves a power of $2$.
 \end{itemize}

	\subsection{Methodology:}
	In this section, we shall explain the methodology that is common 
	in the proof of our main results i.e., Theorems~\ref{main result1 for x^2=By^p+Cz^p Type I},~\ref{main result1 for x^2=By^p+2^rz^p over K} of this article.

	
	For any non-trivial primitive solution $(a, b, c)\in \mcO_{K}^3$ to the equation $x^2=By^p+Cz^p$ (resp., $ 2x^2=By^p+2^rz^p$), consider the Frey elliptic curve $E:=E_{a,b,c}$ as in \eqref{Frey curve for x^2=By^p+Cz^p of Type I}. (resp., \eqref{Frey curve for 2x^2=By^p+Cz^p of Type I}). The following strategy, apart from the technical difficulties, is inspired by the work of Freitas and Siksek in~\cite{FS15}:

	
	\begin{enumerate}
		\item For any non-trivial primitive solution $(a,b,c)\in W_K$ (resp., $ \mcO^3_K$) to the equation $x^2=By^p+Cz^p$ with exponent $p$ of Type I or II (resp., $2x^2=By^p+2^rz^p$), there exists a constant $A:=A_{K,B,C}>0$ (depending on $K,B,C$) such that for primes $p >A$, the Frey elliptic curve $E/K$ is modular (cf. Theorem~\ref{modularity of Frey curve x^2=By^p+Cz^p over W_K}).
		
		The same statement also holds for any non-trivial primitive solution $ (a,b,c)\in \mcO_{K}^3 \setminus S_r \ (\text{with }r=1,2,4,5)$ to $x^2=By^p+Cz^p$ of Type II (cf. Theorem~\ref{modularity of Frey curve of x^2=By^p+2^rz^p over K}).
		
		\item By~\cite[Theorem 2]{FS15 Irred}, for $p \gg 0$, the residual representation $\bar{\rho}_{E,p}$ is irreducible. 

		\item the Frey elliptic curve $E$ has semi-stable reduction away from $S_K^{\prime}$ and satisfies $p | v_\mfq(\Delta_E)$ for $\mfq \notin S_K^{\prime}$ (cf. Theorem~\ref{reduction away from S}).
		
		\item 
	        \begin{enumerate}
	         \item For any non-trivial primitive solution $(a,b,c)\in W_K$ to the equation $x^2=By^p+Cz^p$ with exponent $p \gg 0$ of Type I or II, we get $v_\mfP(j_E) < 0$ and $p \nmid v_\mfP(j_E)$ for $\mfP \in S_K$ (cf. Lemma~\ref{reduction on T and S}).  By (1), (2), (3) and a level lowering result by ~\cite[Theorem 7]{FS15}, there exists a constant $V:=V_{K,B,C}>0$ (depending on $K,B,C$) and an elliptic curve $E'/K$  does not depend on the solution $(a,b,c)$, unlike $E$, such that $\bar{\rho}_{E,p} \sim \bar{\rho}_{E^\prime,p}$ for all $p>V$ (cf. Theorem~\ref{auxilary result x^2=By^p+Cz^p over W_K}). Now, by Lemma~\ref{criteria for potentially multiplicative reduction}, we get $v_\mfP(j_{E^\prime})<0$ for $\mfP \in S_K$. 
	        By using the technique of Mocanu in \cite{M22}, we relate $j_{E'}$ to a solution of~\eqref{keyequation}, together with \eqref{assumption for main result1 for x^2=By^p+Cz^p Type I} to get $v_\mfP(j_{E^\prime}) \geq0$ for some $\mfP \in S_K$, which is a contradiction.

            A similar calculation works for any non-trivial primitive solution $(a,b,c)\in  \mcO_K^3$ to the equation $ 2x^2=By^p+2^rz^p$. 

             \item For any non-trivial primitive solution $ (a,b,c)\in \mcO_{K}^3 \setminus S_r \ (\text{with }r=1,2,4,5)$ to the equation $x^2=By^p+Cz^p$ with exponent $p \gg0$ of Type II, we get either $p | \#\bar{\rho}_{E,p}(I_\mfP)$ or $3 | \#\bar{\rho}_{E,p}(I_\mfP)$ for $\mfP \in U_K$ (cf. Lemma~\ref{reduction on T and S x^2=By^p+Cz^p Type II over K}). 
             Now, arguing as before, there exists an elliptic curve $E'/K$ with $\bar{\rho}_{E,p} \sim \bar{\rho}_{E^\prime,p}$ for all $p>V$ (cf. Theorem~\ref{auxilary result x^2=By^p+2^rz^p over K}). By Lemmas~\ref{criteria for potentially multiplicative reduction}, ~\ref{3 divides discriminant}, we get either $v_\mfP(j_{E^\prime})<0$ or $3\nmid v_\mfP(j_{E^\prime})$ for $\mfP \in U_K$. Then, we relate $j_{E'}$ in terms of solutions of~\eqref{keyequation}, together with~\eqref{assumption for main result1 x^2=By^p+2^rz^p over K}, to get $v_\mfP(j_{E^\prime}) \geq0$ and $3| v_\mfP(j_{E^\prime})$ for some $\mfP \in U_K$, which is a contradiction.

         	 \end{enumerate}

	\end{enumerate}

\subsection{Limitations in~\cite{KS22} and generalizations:}
In this section, we shall mention the limitations of the method in~\cite{KS22} for studying the solutions of $x^p+y^p=z^2$ and explain how to overcome them by employing Mocanu's ideas in~\cite{M22} to $x^2 = By^p+Cz^p$.

\begin{itemize}
 \item The proof of the main result in~\cite[Theorem 5.3]{KS22} for the equation $x^p+y^p=z^2$, depends on some explicit bounds on the solutions of the $S_K$-unit equation and the $S_L$-unit equation, where $L$ varies over certain quadratic extensions of $K$. This helped us to study the asymptotic solutions in $W_K^\prime$, a certain subset of $\mcO_K^3$. 
 
 We could not provide the local criteria of $K$ in~\cite{KS22} for the equation $x^p+y^p=z^2$, as we were working on $S_K$ as well as $S_L$-unit equations, where $L$ varies over certain quadratic extensions of $K$.

 \item On the other hand, the proofs of Theorems~\ref{main result1 for x^2=By^p+Cz^p Type I}, ~\ref{main result1 for x^2=By^p+2^rz^p over K} depend on some explicit bounds on the solutions of \eqref{keyequation} and this aspect gave us an advantage to study the asymptotic solutions of the equation $ x^2=By^p+Cz^p$ in $W_K$ (with $B$ is odd and $C$ is either odd or $2^r$ with $r\in \N$), and in $\mcO_{K}^3 \setminus S_r$ (with $B$ is odd and $C=2^r$ with $r=1,2,4,5$). Note that, in~\cite[Theorem 5.3]{KS22}, the method only works for $W_K^\prime$.

	Since we are working on the solutions of~\eqref{keyequation}, we are able to provide the local criteria of $K$ for Theorems~\ref{main result1 for x^2=By^p+Cz^p Type I}, ~\ref{main result1 for x^2=By^p+2^rz^p over K}.

\item 	
	In Theorem~\ref{auxilary result x^2=By^p+Cz^p over W_K} (resp., ~\cite[Theorem 6.4]{KS22})  we show the existence of an elliptic curve $E'/K$ having a \textbf{non-trivial $2$-torsion point}, having good reduction away from $S_K^{\prime}$ and $v_\mfP(j_{E^\prime})<0$ for $\mfP \in S_K$.
    \begin{itemize}
     \item 	In~\cite{KS22}, we construct a quadratic extension $L$ over $K$ such that $E^\prime/ L$ acquires full $2$-torsion. Then, we relate the $j$-invariant $j_{E'}$ of $E'$ to solutions of $S_K$-unit equation and $S_L$-unit equation to obtain $v_\mfP(j_{E^\prime})\geq 0$ for some $\mfP \in S_K$, 
     which is a contradiction to~\cite[Theorem 6.4]{KS22}. This proves the main result~\cite[Theorem 5.3]{KS22}.
     
     \item In this article, we do not require to look at the solutions of either $S_K$-unit equation or the 
           $S_L$-unit equation, as we relate the $j$-invariant $j_{E'}$ to a solution of~\eqref{keyequation} to get $v_\mfP(j_{E^\prime})\geq 0$ for some $\mfP \in S_K$, 
           which is a contradiction to Theorem~\ref{auxilary result x^2=By^p+Cz^p over W_K}. This proves the main result Theorem~\ref{main result1 for x^2=By^p+Cz^p Type I}.

    \end{itemize}
    \item A similar analysis has been worked out for solutions in $\mcO_K^3 \setminus S_r$ (with $r=1,2,4,5$)
    (cf. Theorem~\ref{main result1 for x^2=By^p+2^rz^p over K}). In this case, we constructed an elliptic curve $E^\prime$ with $v_\mfP(j_{E^\prime})<0$ or $3\nmid v_\mfP(j_{E^\prime})$ for $\mfP \in U_K$ (cf. Theorem~\ref{auxilary result x^2=By^p+2^rz^p over K}).
\end{itemize}    
We end this section with some preliminaries.


%
%

\subsection{Preliminaries:}	
 	\label{section for preliminary}

	 We use the notations $\mcO_K$, $P$, and $\mfn$ to represent the ring of integers, $\mrm{Spec}(\mcO_K)$, and an ideal of $K$, respectively. Let $E/K$ be an elliptic curve of conductor $\mfn$. For any $\mfq \in P$, let $\Delta_\mfq$ be the minimal discriminant of $E$ at $\mfq$. Let 
	
	\begin{equation}
		\label{conductor of elliptic curve}
		\mfm_p:= \prod_{ p|v_\mfq(\Delta_\mfq), \ \mfq ||\mfn} \mfq \text{ and } \mfn_p:=\frac{\mfn}{\mfm_p}.
	\end{equation}
 
	We state a conjecture, which is an extension of the Eichler-Shimura theorem over $\Q$. 
	\begin{conj}[Eichler-Shimura]
		\label{ES conj}
		Let $f$ be a Hilbert modular newform over $K$ of parallel weight $2$, level $\mfn$, and with coefficient field $\Q_f= \Q$. Then, there exists an elliptic curve $E_f /K$ with conductor $\mfn$ having same $L$-function as $f$.
	\end{conj}
    	In~\cite[Theorem 7.7]{D04}, Darmon showed that Conjecture~\ref{ES conj} holds over $K$, if  either 
    $[K: \Q] $ is odd or there exists some $\mfq \in P$ such that $v_\mfq(\mfn) = 1$. 
    In \cite[Corollary 2.2]{FS15}, Freitas and Siksek provided a partial answer to Conjecture~\ref{ES conj} in terms of mod $p$  Galois representations attached to $E$.

	\section{On the solutions of $x^2=By^p+Cz^p$ over $W_K$ and $2x^2=By^p+2^rz^p$ over $K$}
	\label{notations section for x^2=By^p+Cz^p} 
		\label{notations section for 2x^2=By^p+Cz^p} 	
	In this section, we study the solutions of following Diophantine equations:
	\begin{equation}
		\label{x^2=By^p+Cz^p}
		x^2 =By^p+Cz^p,
	\end{equation} 

\begin{equation}
	\label{2x^2=By^p+Cz^p}
	2x^2 =By^p+Cz^p,
\end{equation} 
	with prime exponent $p\geq 3$ and $B,C \in \Z \setminus \{0\}$. Throughout, we assume that $B$ is odd.
	\begin{itemize}
		\item We say the equation~\eqref{x^2=By^p+Cz^p} (resp., \eqref{2x^2=By^p+Cz^p}) with exponent $p$ is of Type I, if $C$ is odd.
		\item We say the equation~\eqref{x^2=By^p+Cz^p} (resp., \eqref{2x^2=By^p+Cz^p}) with exponent $p$ is of Type II, if $C=2^r$ for some $r\in \N$. 
	\end{itemize}
	For $n \in \Z$, define $S_K(n):= \{ \mfP \in P :\ \mfP|2n \}$.
	Let $S_K:=S_K(1)$, $S_K^{\prime}:= S_K(BC)$ and $U_K:=\{ \mfP \in S_K: (3, v_\mfP(2))=1 \}$.

	\begin{dfn}[Trivial solution]
		We say a solution $(a, b, c)\in \mcO_K^3$ to the equation \eqref{x^2=By^p+Cz^p} (resp., \eqref{2x^2=By^p+Cz^p}) with exponent $p$ is trivial, if $abc=0$.
        We say $(a, b, c)\in \mcO_K^3$ is primitive if the ideal generated by $a, b, c$ in $\mcO_K$ is $\mcO_K$.
	\end{dfn}  

	\begin{dfn}
		\label{def for W_K}
			\label{remark for W_K}
	Let $W_K$ be the set of all non-trivial primitive solutions $(a, b, c)\in  \mcO_K^3$ to the equation~\eqref{x^2=By^p+Cz^p} with exponent $p$ of Type I or II with $\mfP |bc$ for every $\mfP \in S_K$. Note that, for any $\mfP \in S_K$ and $(a, b, c)\in W_K$, $\mfP$ divides exactly one of $b$ and $c$. 
\end{dfn}

	\subsection{Main result}
	\label{section for main result of x^2=By^p+2^rz^p} 
	\label{section for main result of 2x^2=By^p+2^rz^p}
	For any set $S \subseteq P$, let $\mcO_{S}:=\{\alpha \in K : v_\mfP(\alpha)\geq 0 \text{ for all } \mfP \in P \setminus S\}$ be the ring of $S$-integers in $K$ and $\mcO_{S}^*$ be the $S$-units of $\mcO_{S}$.  
    Let $\Cl_S(K):= \mrm{Cl}(K)/\langle [\mfP]\rangle_{\mfP\in S}$ and $\Cl_S(K)[n]$ be its $n$-torsion points. 

	We now show that the equation~\eqref{x^2=By^p+Cz^p} (resp., \eqref{2x^2=By^p+Cz^p}) with exponent $p$ of Type I or II (resp., Type II) has no asymptotic solution in $W_K$ (resp., $\mcO_K^3$). More precisely;
	\begin{thm}
		\label{main result1 for x^2=By^p+Cz^p Type I}
		\label{main result1 for x^2=By^p+Cz^p Type II}
		\label{main result1 for 2x^2=By^p+Cz^p Type I}
		\label{main result1 for 2x^2=By^p+Cz^p Type II}
		Let $K$ be a totally real field with $\Cl_{S_K^{\prime}}(K)[2]=1$. Suppose for every solution $(\alpha, \beta, \gamma) \in \mcO_{S_K^{\prime}}^\ast \times \mcO_{S_K^{\prime}}^\ast \times \mcO_{S_K^{\prime}}$ to  $\alpha+\beta=\gamma^2$,
        there exists $\mfP \in S_K$ that satisfies 
		\begin{equation}
			\label{assumption for main result1 for x^2=By^p+Cz^p Type I}
	\left| v_\mfP \left(\alpha \beta^{-1}\right) \right| \leq 6v_\mfP(2).
		\end{equation}
		Then, the Diophantine equation \eqref{x^2=By^p+Cz^p} (resp., \eqref{2x^2=By^p+Cz^p}) with exponent $p$ of Type I or II (resp., Type II) has no asymptotic solution in $W_K$ (resp., $\mcO_K^3$).
\end{thm}

\begin{remark}
	\label{def for W_K'}
	\label{remark for W_K'}
	There are no non-trivial primitive solutions $(a, b, c)\in  \mcO_K^3$ to the equation $2x^2=By^p+Cz^p$ with exponent $p > [K:\Q]$ with $\mfP |bc$ for every $\mfP \in S_K$, where $B,C$ are odd integers. Indeed, let $(a, b, c)$ be a solution such that $\mfP |bc$ for every $\mfP \in S_K$. Since $B,C$ are odd, $\mfP $ divides both $b$ and $c$, which implies $\mfP^p | Bb^p+Cc^p=2a^2$. Since $p>[K:\Q]$, we have $\mfP |a$, which is a contradiction to the fact that $(a,b,c)$ is primitive. 
\end{remark}

By~\cite[Theorem 39]{M22}, for any finite set $S \subseteq P$, the equation $\alpha+\beta=\gamma^2$ with $(\alpha, \beta, \gamma) \in \mcO_{S}^\ast \times \mcO_{S}^\ast \times \mcO_{S}$ has only finitely many solutions. The following proposition, whose proof will be given at the end of this section, is a consequence of Theorem~\ref{main result1 for x^2=By^p+Cz^p Type I} and will be relevant in \S\ref{section for local criteria for Diophantine equations over $W_K$}. We say that $S \subseteq P$ is principal if $\mfP$ is principal for every $\mfP \in S$.
\begin{prop}
	\label{main result2 for x^2=By^p+Cz^p Type I}
	\label{main result2 for x^2=By^p+Cz^p Type II}
	\label{main result2 for 2x^2=By^p+Cz^p Type I}
	\label{main result2 for 2x^2=By^p+Cz^p Type II}
	Let $K$ be a field such that $S_K^{\prime}=S_K$ is principal and $2 \nmid h_K$. Suppose for every solution $(\alpha, \gamma) \in \mcO_{S_K^{\prime}}^\ast \times \mcO_{S_K^{\prime}}$ to  $\alpha+1=\gamma^2$,
	there exists $\mfP \in S_K$ that satisfies
	\begin{equation}
	\label{assumption for main result2 for x^2=By^p+Cz^p Type I}
	|v_\mfP(\alpha)|\leq 6v_\mfP(2).
	\end{equation} 
Then, the Diophantine equation \eqref{x^2=By^p+Cz^p} (resp., \eqref{2x^2=By^p+Cz^p}) with exponent $p$ of Type I or II (resp., Type II) has no asymptotic solution in $W_K$ (resp., $\mcO_K^3$).
\end{prop}

	\subsection{Steps to prove Theorem~\ref{main result1 for x^2=By^p+Cz^p Type I}}
	\begin{itemize}
		\item For any non-trivial and primitive solution $(a, b, c)\in \mcO_K^3$ to the equation \eqref{x^2=By^p+Cz^p} with exponent $p$, the Frey curve $E:=E_{a,b,c}$ is given by
		\begin{equation}
			\label{Frey curve for x^2=By^p+Cz^p of Type I}
			E:=E_{a,b,c} : Y^2 = X(X^2+2aX+Bb^p),
		\end{equation}
		with $c_4=2^4(Bb^p+4Cc^p),\ \Delta_E=2^{6}(B^2C)(b^2c)^{p}$ and $ j_E=2^{6} \frac{(Bb^p+4Cc^p)^3}{B^2C(b^2c)^{p}}$, where $j_E$ (resp., $\Delta_E$) denote the $j$-invariant (resp., discriminant) of $E$.
		\item For any non-trivial and primitive solution $(a, b, c)\in \mcO_K^3$ to the equation \eqref{2x^2=By^p+Cz^p} with exponent $p$ of Type II, the Frey curve $E:=E_{a,b,c}$ is given by
		
			\begin{equation}
				\label{Frey curve for 2x^2=By^p+Cz^p of Type I}
				E=E_{a,b,c} : Y^2 = X(X^2-4aX+2Bb^p),
			\end{equation}
			with $c_4=2^5(Bb^p+2^{r+2}c^p),\ \Delta_E=2^{9+r}B^2(b^2c)^{p}$ and $ j_E=2^{6-r} \frac{(Bb^p+2^{r+2}c^p)^3}{B^2(b^2c)^{p}}$.
	\end{itemize}
	 We now prove the modularity of the Frey curve $E:=E_{a,b,c}$ in ~\eqref{Frey curve for x^2=By^p+Cz^p of Type I} (resp., ~\eqref{Frey curve for 2x^2=By^p+Cz^p of Type I}) associated to $(a,b,c)\in W_K$ (resp., $ \mcO_K^3$).
	\begin{thm}
		\label{modularity of Frey curve x^2=By^p+Cz^p over W_K}
		Let $(a,b,c)\in W_K$ (resp., $\mcO_K^3$) be a non-trivial primitive solution to the equation~\eqref{x^2=By^p+Cz^p} (resp., \eqref{2x^2=By^p+Cz^p}) with exponent $p$ of Type I or II (resp., Type II). Let $E:=E_{a,b,c}$ be the Frey curve attached to $(a,b,c)$ as in~\eqref{Frey curve for x^2=By^p+Cz^p of Type I} (resp., ~\eqref{Frey curve for 2x^2=By^p+Cz^p of Type I}). Then, there exists a constant $A:=A_{K,B,C}>0$ (depending on $K,B,C$) such that for primes $p >A$, $E/K$ is modular.
	\end{thm}
	
	\begin{proof}
        By ~\cite[Theorem 5]{FLBS15},
        there exist finitely many elliptic curves over $K$ (up to $\bar{K}$-isomorphism) which are not modular. Let $j_1,\ldots,j_t \in K$ be the $j$-invariants of these.
        
        \begin{itemize}
         \item  Suppose $(a,b,c)\in W_K$, and $E:=E_{a,b,c}$ is the Frey curve attached to $(a,b,c)$ as in~\eqref{Frey curve for x^2=By^p+Cz^p of Type I}. Then $j$-invariant $ j_E=2^{6} \frac{(Bb^p+4Cc^p)^3}{B^2C(b^2c)^{p}}=2^{6}\frac{(4-\lambda(E))^3}{\lambda(E)^2}$ for $\lambda(E)= -\frac{Bb^p}{Cc^p}$.
		For $i=1,2,\ldots,t$, the equation $j_E=j_i$ has at most three solutions in $K$. Hence, there exist
		$\lambda_1, \lambda_2, ..., \lambda_m \in K$ with $m\leq 3t$ such that $E$ is modular for all $\lambda(E) \notin\{\lambda_1, \lambda_2, ..., \lambda_m\}$.
		If $\lambda(E)= \lambda_k$ for some $k \in \{1, 2, \ldots, m \}$, then $\left(\frac{b}{c} \right)^p=\frac{-C\lambda_k}{B}$.
		This equation determines $p$ uniquely, denoting it by $p_k$. Suppose $p \neq q$ are primes such that $\left(\frac{b}{c} \right)^p=\left(\frac{b}{c} \right)^q$, which means $\left(\frac{b}{c}\right)$ is a root of unity. Since $K$ is totally real, we get $b=\pm c$. For $\mfP \in S_K$, $\mfP \mid bc$ implies that $\mfP |a$,
		contradicts $(a,b,c) \in W_K$. Now, the proof of the theorem follows by taking $A=\max \{p_1,...,p_m\}$.

		\item Suppose $(a,b,c)\in\mcO_K^3$, and $E:=E_{a,b,c}$ is the Frey curve attached to $(a,b,c)$ as in~\eqref{Frey curve for 2x^2=By^p+Cz^p of Type I}. Then $j_E=2^{6-r} \frac{(Bb^p+2^{r+2}c^p)^3}{B^2(b^2c)^{p}}$.
		Similar to the previous paragraph, there exist $\lambda_k \in K$ with $1 \leq k \leq m$ such that $E/K$ is modular for all $\lambda(E) \notin\{\lambda_1, \lambda_2, ..., \lambda_m\}$.
		If $\lambda(E)= \lambda_k$ for some $k \in \{1, 2, \ldots, m \}$, then $\left(\frac{b}{c} \right)^p= \frac{-2^r\lambda_k}{B}$.
		The above equation determines $p$ uniquely, denoting it by $p_k$. If not, since $K$ is totally real,
		we get $b=\pm c$.
		%
	Since $2a^2= Bb^p+2^rc^p$ for some $r \in \N$ and $B$ is odd, we get $\mfP|b$ for any $\mfP \in S_K$. Since $b=\pm c$, $\mfP^p |Bb^p+2^rc^p=2a^2$. Assume $p > [K: \Q]$. Then we get $\mfP | a$, which contradicts $(a,b,c)$ is primitive. Now the proof of the theorem follows by taking $A=\max \{p_1,...,p_m, [K: \Q]+1\}$.
        \end{itemize}

	\end{proof}

	\subsection{Reduction type}

	The following lemma characterizes the type of reduction of the Frey curve $E:= E_{a,b,c}$ in~\eqref{Frey curve for x^2=By^p+Cz^p of Type I} (resp., ~\eqref{Frey curve for 2x^2=By^p+Cz^p of Type I}) 
	at primes $\mfq$ away from $S_K^{\prime}$.
	\begin{lem}
		\label{reduction away from S}
		Let $(a,b,c) \in \mcO_K^3$ be a non-trivial primitive solution to the equation~\eqref{x^2=By^p+Cz^p} (resp., \eqref{2x^2=By^p+Cz^p}) with exponent $p$ of Type I or II (resp., Type II). Let $E$ be the Frey curve attached to $(a,b,c)$ as in~\eqref{Frey curve for x^2=By^p+Cz^p of Type I} (resp., ~\eqref{Frey curve for 2x^2=By^p+Cz^p of Type I}). Then at all primes $\mfq \notin S_K^{\prime}$, $E$ is minimal, semi-stable at $\mfq$ and satisfies $p | v_\mfq(\Delta_E)$. Let $\mfn$ be the conductor of $E$, and $\mfn_p$ be as in \eqref{conductor of elliptic curve}. Then,
			\begin{equation}
				\label{conductor of E and E' x^2=By^p+Cz^p Type I}
				\mfn=\prod_{\mfP \in S_K^{\prime}}\mfP^{r_\mfP} \prod_{\mfq|bc,\ \mfq \notin S_K^{\prime}}\mfq,\ \mfn_p=\prod_{\mfP \in S_K^{\prime}}\mfP^{r_\mfP^{\prime}},
			\end{equation}
			where $0 \leq r_\mfP^{\prime} \leq r_\mfP $ with $r_\mfP 
			\leq 2+6v_\mfP(2)$ for $\mfP |2$ and $ r_\mfP\leq 2+3v_\mfP(3)$ for $\mfP \nmid 2$.
\end{lem}
	
	\begin{proof} We give a proof of this lemma in two cases, as we did previously.

     \begin{itemize}
      \item Suppose $(a,b,c) \in \mcO_K^3$ is a non-trivial primitive solution to the equation~\eqref{x^2=By^p+Cz^p} of Type I or II. Then $c_4=2^4(Bb^p+4Cc^p)$ and $\Delta_E=2^{6}B^2C(b^2c)^{p}. $

	Let $\mfq \in P \setminus S_K^\prime$. If $\mfq \not|\Delta_E$, then $E$ has good reduction at $\mfq$ and $p | v_\mfq(\Delta_E)=0$.
        If $\mfq|\Delta_E$, then $\mfq$ divides precisely one of $b$ and $c$, since $(a,b,c)$ is primitive and $\mfq \nmid 2BC$. This implies that $\mfq\nmid c_4$, hence $E$ is minimal and has multiplicative reduction at $\mfq$. Since $v_\mfq(\Delta_E)=p v_\mfq(b^2c) $, $p | v_\mfq(\Delta_E)$. By the definition of $\mfn_p$ in~\eqref{conductor of elliptic curve}, we get $\mfq \nmid \mfn_p$ for all $\mfq \notin S_K^{\prime}$. Finally, for $\mfP \in S_K^{\prime}$, the bounds on $r_\mfP$ follow from \cite[Theorem IV.10.4]{S94}.

      \item
        Suppose $(a,b,c) \in \mcO_K^3$ is a non-trivial primitive solution to the equation~\eqref{2x^2=By^p+Cz^p} of Type II. Then $c_4=2^5(Bb^p+2^{r+2}c^p)$ and $ \Delta_E=2^{9+r}B^2(b^2c)^{p}$. The rest of the argument is similar to the case above. Hence, the proof of the lemma follows.
     \end{itemize}
	\end{proof}
%

	\subsubsection{Type of reduction with image of inertia}
    For any elliptic curve $E/K$, let $\bar{\rho}_{E,p} : G_K:=\Gal(\overline{K}/K) \rightarrow \mathrm{Aut}(E[p]) \simeq \GL_2(\F_p)$ be the residual Galois representation of $G_K$, induced by the action of $G_K$ on  $E[p]$, the $p$-torsion of $E$.
    We first recall ~\cite[Lemmas 3.4, 3.6]{FS15}, which will be useful for determining the types of the reduction of the Frey curve at $\mfP \in P$. 
	\begin{lem}
		\label{criteria for potentially multiplicative reduction}
		Let $E/K$ be an elliptic curve and $p>5$ be a prime. For $\mfq \in P$ with $\mfq \nmid p$, $E$ has potentially multiplicative reduction at $\mfq$ and $p \nmid v_\mfq(j_E)$ if and only if $p | \# \bar{\rho}_{E,p}(I_\mfq)$.
	\end{lem}
	\begin{lem}
		\label{3 divides discriminant}
		Let $E/K$ be an elliptic curve and $p\geq 3$ be a prime.
		Suppose $E$ has potential good reduction at $\mfP$ for some $\mfP \in S_K$. Then, $3 \nmid v_\mfP(\Delta_E)$ if and only if $3 | \#\bar{\rho}_{E,p}(I_\mfP)$.
	\end{lem}
	
   The following lemma determines the type of reduction of $E_{a,b,c}$ at primes $\mfq \nmid 2pBC$.
	\begin{lem}
		\label{Type of reduction at q away from 2,p,B x^2=By^p+Cz^p}
		Let $(a,b,c)\in \mcO_K^3$ be a non-trivial primitive solution to the equation~\eqref{x^2=By^p+Cz^p} (resp., \eqref{2x^2=By^p+Cz^p}) with exponent $p>5$ of Type I or II (resp., Type II), and let $E$ be the associated Frey curve. Suppose $\mfq \in P$ with $\mfq\nmid 2pBC$. Then $p \nmid \#\bar{\rho}_{E,p}(I_\mfq)$.
 	\end{lem}
	
	\begin{proof}
		Suppose $(a,b,c)\in \mcO_K^3$ is a non-trivial primitive solution to the equation~\eqref{x^2=By^p+Cz^p} of Type I or II.
		By Lemma~\ref{criteria for potentially multiplicative reduction}, it is enough to show that
		either $v_\mfq(j_E) \geq 0$ or $p | v_\mfq(j_E)$. Recall that $\Delta_E=2^{6}(B^2C)(b^2c)^{p}$ and $c_4=2^4(Bb^p+4Cc^p)$. 
		If $\mfq \nmid \Delta_E$, then $E$ has good reduction at $\mfq$, and hence $v_\mfq(j_E)\geq 0$. If $\mfq | \Delta_E$ then $\mfq |bc$,
		and hence $\mfq$ divides exactly one of $b$ and $c$. Therefore, $\mfq \nmid c_4$ and $p |v_\mfq(j_E)=-pv_\mfq(b^2c)$, which completes the proof.

	    Suppose $(a,b,c)\in \mcO_K^3$ is a non-trivial primitive solution to the equation~\eqref{2x^2=By^p+Cz^p} of Type II. Then, the proof follows if we argue along the lines of the case above.
	\end{proof}
	
	We discuss the type of reduction of $E_{a,b,c}$ in~\eqref{Frey curve for x^2=By^p+Cz^p of Type I} (resp., ~\eqref{Frey curve for 2x^2=By^p+Cz^p of Type I}) at $\mfP \in S_K$.
	\begin{lem}
		\label{reduction on T and S}
		Let $(a,b,c)\in W_K$ (resp., $\mcO_K^3$) be a non-trivial primitive solution to the equation~\eqref{x^2=By^p+Cz^p} (resp., ~\eqref{2x^2=By^p+Cz^p}) with exponent $p > 6v_\mfP(2)+v_\mfP(C)$ (resp., $p > \max \{ (6+r)v_\mfP(2),\ [K : \Q]\}$) of Type I or II (resp., Type II). Let $E:=E_{a,b,c}$ be the associated Frey curve.
		For $\mfP \in S_K$, we have $\ v_\mfP(j_E) < 0$ and $p \nmid v_\mfP(j_E)$, equivalently $p | \#\bar{\rho}_{E,p}(I_\mfP)$.
    \end{lem} 
    \begin{proof}
We now give a proof of this lemma in two cases, as we did previously.
    \begin{itemize}
         \item Let $(a,b,c)\in W_K$ be a solution to the equation~\eqref{x^2=By^p+Cz^p}. Then $\mfP$ divides exactly one of $b$ and $c$.

    Suppose the equation~\eqref{x^2=By^p+Cz^p} with exponent $p$
    is of Type I. Recall that $j_E=2^{6} \frac{(Bb^p+4Cc^p)^3}{B^2C(b^2c)^{p}}$. If $\mfP |b$, then $\mfP \nmid c$.
    Since $p > 6v_\mfP(2)$, $v_\mfP(j_E)= 6v_\mfP(2)+6v_\mfP(2)-2pv_\mfP(b)= 2(6v_\mfP(2)-pv_\mfP(b))$, $v_\mfP(j_E) <0$ and $p \nmid v_\mfP(j_E)$. Similar proof works for $\mfP |c$ as well.

    Suppose the equation~\eqref{x^2=By^p+Cz^p} with exponent $p$
    is of Type II, i.e., $C=2^r$ for some $r \in \N$. If $\mfP |b$, then $v_\mfP(j_E)=  (6-r)v_\mfP(2)+3(r+2)v_\mfP(2)-2pv_\mfP(b)= (12+2r)v_\mfP(2)-2pv_\mfP(b) <0$ and $p \nmid v_\mfP(j_E)$, since $p > (6+r)v_\mfP(2)$. If $\mfP |c$ then $v_\mfP(j_E)=  (6-r)v_\mfP(2)-pv_\mfP(c) <0$ since $p > (6+r)v_\mfP(2)$.
		 Since $p > (6+r)v_\mfP(2)> (6-r)v_\mfP(2) \geq0$ for $1 \leq r\leq 6$ and $-p< (-6-r)v_\mfP(2)<(6-r)v_\mfP(2)<0$ for $r> 6$, we get $p \nmid v_\mfP(j_E)$.
		Hence, by Lemma~\ref{criteria for potentially multiplicative reduction}, we get $p | \#\bar{\rho}_{E,p}(I_\mfP)$.

\item Let $(a,b,c)\in \mcO_K^3$ be a non-trivial primitive solution to the equation~\eqref{2x^2=By^p+Cz^p} with exponent $p$ of Type II, and let $\mfP \in S_K$. Then $\mfP |b$. Since $p>[K: \Q] $,   $\mfP \nmid c$. Recall $j_E=2^{6-r} \frac{(Bb^p+2^{r+2}c^p)^3}{B^2(b^2c)^{p}}$. Since $p > (6+r)v_\mfP(2)$, $v_\mfP(j_E)= (12+2r)v_\mfP(2)-2pv_\mfP(b) <0$ and $p \nmid v_\mfP(j_E)$.
Hence, by Lemma~\ref{criteria for potentially multiplicative reduction}, we get $p | \#\bar{\rho}_{E,p}(I_\mfP)$.

    \end{itemize}

	\end{proof}

	\subsection{Proof of Theorem~\ref{main result1 for x^2=By^p+Cz^p Type I}.}
%
The proof of this theorem depends on the following result.	
	\begin{thm}
		\label{auxilary result x^2=By^p+Cz^p over W_K}
		Let $K$ be a totally real field. Then, there is a constant $V:=V_{K,B,C}>0$ (depending on $K,B,C$) such that the following holds.
		Let $(a,b,c)\in W_K$ (resp., $ \mcO_K^3$) be a non-trivial primitive solution to the equation~\eqref{x^2=By^p+Cz^p} (resp., ~\eqref{2x^2=By^p+Cz^p}) with exponent $p>V$ of Type I or II (resp., Type II). Let $E$ be the Frey curve as in \eqref{Frey curve for x^2=By^p+Cz^p of Type I} (resp., \eqref{Frey curve for 2x^2=By^p+Cz^p of Type I}). Then, there exists an elliptic curve $E^\prime/K$ such that:
		\begin{enumerate}
			\item $E^\prime$ has good reduction away from $S_K^{\prime}$ and has a non-trivial $2$-torsion point;
			\item $\bar{\rho}_{E,p} \sim\bar{\rho}_{E^\prime,p}$, and  $v_\mfP(j_{E^\prime})<0$ for $\mfP \in S_K$.
		\end{enumerate}
	\end{thm}

    \begin{proof}[Proof of Theorem~\ref{auxilary result x^2=By^p+Cz^p over W_K}] 
		By Theorem~\ref{modularity of Frey curve x^2=By^p+Cz^p over W_K}, $E$ is modular for primes $p>A:=A_{K,B,C}$ with $A \gg 0$. By Lemma~\ref{reduction away from S}, $E$ is semi-stable away from $S_K^{\prime}$. If necessary, we can take the Galois closure of $K$ to ensure that $\bar{\rho}_{E,p}$ is irreducible for $p \gg 0$ (cf.~\cite[Theorem 2]{FS15 Irred}).
		
			By ~\cite[Theorem 7]{FS15}, there exists a Hilbert
        modular newform $f$ of parallel weight $2$, level $\mfn_p$ and some prime $\omega$ of $\Q_f$ such that $\omega | p$ and $\bar{\rho}_{E,p} \sim \bar{\rho}_{f,\omega}$ for $p \gg 0$, where $\bar{\rho}_{f,\omega}$ denotes the residual Galois representation attached to $f,\omega$. By allowing $p$ to be sufficiently large, we can assume  $\Q_f=\Q$ (cf. ~\cite[\S 4]{FS15} for more details).

		Let $\mfP \in S_K$. Then $E$ has potential multiplicative reduction at $\mfP$ and $p | \#\bar{\rho}_{E,p}(I_\mfP)$ for $p \gg 0$
		(cf. Lemma~\ref{reduction on T and S}). The existence of $E_f$ then follows from \cite[Corollary 2.2]{FS15} for all $p\gg 0$ after leaving primes $p$ with $p \mid \left( \text{Norm}(K/\Q)(\mfP) \pm 1 \right)$. Therefore, $\bar{\rho}_{E,p} \sim \bar{\rho}_{E_f,p}$ for some elliptic curve $E_f$ with conductor $\mfn_p$ for $p>V:=V_{K,B,C}$, where $V_{K,B,C}$ is the maximum of all the above implicit/explicit lower bounds.
		
		\begin{itemize}
			\item Since the conductor of $E_f$ is $\mfn_p$ given in \eqref{conductor of E and E' x^2=By^p+Cz^p Type I}, $E_f$ has good reduction away from $S_K^{\prime}$. Now, arguing as in~\cite[page 1247]{M22}, we can enlarge the constant $V$ and by possibly replacing $E_f$ with an isogenous curve, say $E^\prime$, we get $E^\prime/ K$ has a non-trivial $2$-torsion point. Since $E_f \sim E^\prime$, $E^\prime$ has good reduction away from $S_K^{\prime}$. 

			\item Since $E_f$ is isogenous to $E^\prime$ and $\bar{\rho}_{E,p} \sim \bar{\rho}_{E_f,p}$, we get $\bar{\rho}_{E,p} \sim \bar{\rho}_{E^\prime,p}$. As a result, we obtain $p | \# \bar{\rho}_{E,p}(I_\mfP)=\# \bar{\rho}_{E^\prime,p}(I_\mfP)$ for any $\mfP \in S_K$. Finally, by Lemma~\ref{criteria for potentially multiplicative reduction}, we have $v_\mfP(j_{E^\prime})<0$ for any $\mfP \in S_K$.
			\end{itemize}
		This completes the proof of the theorem.
	\end{proof}
	We now present our proof of Theorem~\ref{main result1 for x^2=By^p+Cz^p Type I}, which is inspired from that of~\cite[Theorem 3]{M22}.
	\begin{proof}[Proof of Theorem~\ref{main result1 for x^2=By^p+Cz^p Type I}]
		Suppose $(a,b,c)\in W_K$ (resp., $ \mcO_K^3$) is a non-trivial primitive solution to the equation~\eqref{x^2=By^p+Cz^p} (resp., ~\eqref{2x^2=By^p+Cz^p}) with exponent $p>V$ of Type I or Type II (resp., Type II), where $V:=V_{K,B,C}$ is the constant as in Theorem~\ref{auxilary result x^2=By^p+Cz^p over W_K}. By Theorem~\ref{auxilary result x^2=By^p+Cz^p over W_K}, there exists an elliptic curve $E^\prime/K$ having a non-trivial $2$-torsion point and good reduction away from $S_K^{\prime}$.
Then the elliptic curve $E^\prime/K$ has a model of the form
		\begin{equation}
			\label{j invariant of E'}
		E^\prime:y^2=x^3+cx^2+dx
		\end{equation}		
	for some $c,d \in K$, with $j$-invariant $j_{E^\prime}=2^8 \frac{(c^2-3d)^3}{d^2(c^2-4d)}$. Since $E^\prime$ has good reduction away from $S_K^{\prime}$, we have $j_{E^\prime} \in \mcO_{S_K^{\prime}}$. 
	
	Take $\lambda := \frac{c^2}{d}$ and $\mu := \lambda-4 \in \mcO_{S_K^{\prime}}^\ast$
	(cf.~\cite[Lemma 16(i)]{M22}). By~\cite[Lemma-17(i)]{M22},
	we get $ \lambda  \mcO_K=I^2J$ for some fractional ideal $I$ and $S_K^{\prime}$-ideal $J$. Since $J$ is a $S_K^{\prime}$-ideal, we have $1= [I]^2 \in \Cl_{S_K^{\prime}}(K)$. By hypothesis $\Cl_{S_K^{\prime}}(K)[2]=1$ which gives $I=\gamma I_1$ for some
	$\gamma \in \mcO_K$ and $S_K^{\prime}$-ideal $I_1$. Thus, $\lambda \mcO_K=\gamma^2I_1^2J$ and hence $(\frac{\lambda}{\gamma^2}) \mcO_K$ is an $S_K^{\prime}$-ideal. Therefore, $u=\frac{\lambda}{\gamma^2} \in \mcO_{S_K^{\prime}}^\ast$.
	Now, divide the equation $\mu +4=\lambda$ by $u$ to obtain 
	$\alpha +\beta =\gamma^2$, where $\alpha= \frac{\mu}{u} \in \mcO_{S_K^{\prime}}^\ast$ and $\beta =\frac{4}{u} \in \mcO_{S_K^{\prime}}^\ast$, which implies $ \alpha \beta^{-1}=\frac{\mu}{4}$.
	By~\eqref{assumption for main result1 for x^2=By^p+Cz^p Type I}, there exists $\mfP \in S_K$ with $|v_\mfP(\alpha \beta^{-1})|= |v_\mfP(\frac{\mu}{4} )| \leq 6 v_\mfP(2)$. This means
	\begin{equation}
		\label{inequality for valution of mu}
		-4v_\mfP(2) \leq v_\mfP(\mu) \leq 8 v_\mfP(2).
	\end{equation}
    We now show that the bounds on $v_\mfP(\mu)$ would imply that $v_\mfP(j_{E^\prime}) \geq 0$. 
	Write $j_{E^\prime}$ in terms of $\mu$ yields $j_{E^\prime}= 2^8 \frac{(\mu+1)^3}{\mu}$,
	which means $v_\mfP(j_{E^\prime})=8v_\mfP(2)+3 v_\mfP(\mu +1)-v_\mfP(\mu)$. 
\begin{itemize}
 \item If $v_\mfP(\mu) <0$, then $v_\mfP(\mu +1)=v_\mfP(\mu)$. 
       By \eqref{inequality for valution of mu}, we get $v_\mfP(j_{E^\prime})\geq 0$.
 \item  If $v_\mfP(\mu)=0$, then $v_\mfP(\mu +1)\geq 0$, hence $v_\mfP(j_{E^\prime})\geq 8v_\mfP(2) \geq 0$.
 \item  If $v_\mfP(\mu) >0$, then $v_\mfP(\mu +1)=0$.
        By \eqref{inequality for valution of mu}, we get $v_\mfP(j_{E^\prime})=8v_\mfP(2)- v_\mfP(\mu) \geq 0$. 
\end{itemize}
In all cases, we get $v_\mfP(j_{E^\prime}) \geq 0$, which is a contradiction to Theorem~\ref{auxilary result x^2=By^p+Cz^p over W_K}. This completes the proof of the theorem.
\end{proof}
\begin{proof}[Proof of Proposition~\ref{main result2 for x^2=By^p+Cz^p Type I}]
	By Theorem~\ref{main result1 for x^2=By^p+Cz^p Type I}, it suffices to show that for every solution $(\alpha, \beta, \gamma)\in  \mcO_{S_K}^\ast \times \mcO_{S_K}^\ast \times \mcO_{S_K}$ to the equation $\alpha+\beta=\gamma^2$, there exists $\mfP\in S_K$ such that $|v_\mfP(\alpha\beta^{-1}) |\leq 6v_\mfP(2)$. 
	Let $\mfP \in S_K$. If necessary, by scaling even powers of $\mfP$ and swapping $\alpha, \beta$, we can assume $0\leq v_\mfP(\beta)\leq v_\mfP(\alpha)$ with $v_\mfP(\beta)=0$ or $1$. 
	\begin{enumerate}
		\item Suppose $v_\mfP(\beta)=1$ for some $\mfP \in S_K$. If $v_\mfP(\alpha)>6v_\mfP(2)>1$, then $v_\mfP(\gamma^2)=v_\mfP(\alpha+\beta)=v_\mfP(\beta)=1$, which cannot occur since $v_\mfP(\gamma^2)$ is even. The inequality $v_\mfP(\alpha)\leq 6v_\mfP(2)$ implies 
		$\left|v_\mfP(\alpha \beta^{-1}) \right|\leq 6v_\mfP(2)-1 <6v_\mfP(2)$.
		
		\item Suppose $v_\mfq(\beta)=0$ for all $\mfq \in S_K$, i.e., $\beta$ is a unit in $K$.
		\begin{itemize}
		 \item If $\beta$ is a square, then divide the equation $\alpha +\beta= \gamma^2$ by $\beta$ to obtain an equation of the form $\alpha'+1=\gamma'^2$, where $\alpha'=  \alpha \beta^{-1} \in \mcO_{S_K^{\prime}}^\ast$ and $\gamma' \in \mcO_{S_K^{\prime}}$. By~\eqref{assumption for main result2 for x^2=By^p+Cz^p Type I}, we obtain $|v_\mfP(\alpha \beta^{-1}) |=|v_\mfP(\alpha')|\leq 6v_\mfP(2)$ for some $\mfP \in S_K$.

		 \item Suppose $\beta$ is not a square. If $v_\mfP(\alpha)\leq 6v_\mfP(2)$ for some $\mfP \in S_K$, then we are done. Otherwise, $v_\mfq(\alpha)> 6v_\mfq(2)>1$ for all $\mfq \in S_K$. This gives $\alpha \equiv 0 \pmod {2^6}$ and $\gamma^2=\alpha+\beta \equiv \beta\pmod {2^6}$.
		 Since $v_\mfq(\gamma^2)= v_\mfq(\alpha+\beta)=0$ for all $\mfq \in S_K$ and $S_K=S_K^\prime$, we get $\gamma \in \mcO_K$. 

		 Now, consider the field $L=K(\theta)$, where $\theta:= \frac{\gamma - \sqrt{\beta}}{2}$. The minimal polynomial of $\theta$ is $m_\theta(x)= x^2-\gamma x+  \frac{\gamma^2 - \beta}{4}$. Then $m_\theta(x) \in \mcO_K[x]$ with discriminant $\beta$. 
		 Therefore, $L$ is an everywhere unramified extension of degree 2 over $K$, implying $2| h_K$, which contradicts our hypothesis that $2\nmid h_K$.
        \end{itemize}
    \end{enumerate}
This completes the proof of the proposition.
\end{proof}

\section{Solutions of the Diophantine equation $x^2=By^p+2^rz^p$ over $K$}
\label{section for $x^2=By^p+2^rz^p$ and $2x^2=By^p+2^rz^p$ over $K$} 
In this section, we examine the solutions of the Diophantine equation \eqref{x^2=By^p+Cz^p} with exponent $p$ of Type II, i.e., $x^2 =By^p+2^rz^p$ over $K$. Here, $S_K^{\prime}=\{\mfP \in P : \mfP |2B \}$. Let $h_K^+$ be the narrow class number of $K$. We follow notation as in \S\ref{notations section for x^2=By^p+Cz^p}.
\subsection{Main result}
We write $(ES)$ for ``either $[K: \Q]\equiv 1 \pmod 2$ or Conjecture \ref{ES conj} holds for $K$".
For $r\in \N$, let $S_r:=\{ (\pm \sqrt{2^r+B},1,1), \ (\pm \sqrt{2^r-B},-1,1), \\
(\pm \sqrt{-2^r+B},1,-1),\ (\pm \sqrt{-2^r-B},1,1)\}$. 
For $r\in \{1,2,4,5\}$, we show that the equation \eqref{x^2=By^p+Cz^p} with exponent $p$ of Type II has no asymptotic solution in $\mcO_K^3 \setminus S_r$. More precisely;

\begin{thm}
	\label{main result1 for x^2=By^p+2^rz^p over K}
	Let $K$ be a totally real field satisfying $(ES)$ with $\Cl_{S_K^{\prime}}(K)[2]=1$. Suppose for every solution $(\alpha, \beta, \gamma) \in \mcO_{S_K^{\prime}}^\ast \times \mcO_{S_K^{\prime}}^\ast \times \mcO_{S_K^{\prime}}$ to $\alpha+\beta=\gamma^2$, 	
	there exists $\mfP \in U_K$ that satisfies 
	\begin{equation}
		\label{assumption for main result1 x^2=By^p+2^rz^p over K}
		\left|v_\mfP( \alpha \beta^{-1}) \right| \leq 6v_\mfP(2) \text{ and } v_\mfP \left(\alpha \beta^{-1} \right)\equiv 0 \pmod 3.
	\end{equation}
	Then for $r \in \{1,2,4,5\}$, the Diophantine equation~\eqref{x^2=By^p+Cz^p} with exponent $p$ of Type II has no asymptotic solution in $\mcO_K^3 \setminus S_r$.
\end{thm}	
The following proposition is a consequence of Theorem~\ref{main result1 for  x^2=By^p+2^rz^p over K} and will be relevant in \S\ref{section for local criteria for Diophantine equations over $K$}.
\begin{prop}
	\label{main result2 for x^2=By^p+2^rz^p over K}
	\label{cor to main result2 for x^2=By^p+2^rz^p over K}
	Let $K$ be a field satisfying $(ES)$ with degree $n>1$ and $2 \nmid h_K^+$. Assume $B= \pm1$, and $2$ is inert in $K$. Suppose for every solution $(\alpha, \gamma) \in \mcO_{S_K}^\ast \times \mcO_{S_K}$ to $\alpha+1=\gamma^2$, there exists $\mfP \in U_K$ that satisfies
	\begin{equation}
		\label{assumption for main result2 x^2=By^p+2^rz^p over K}
		\left|v_\mfP(\alpha) \right| \leq 6v_\mfP(2) \text{ and } v_\mfP \left(\alpha \right)\equiv 0 \pmod 3.
	\end{equation}
	Then for $r \in \{1,2,4,5\}$, the Diophantine equation~\eqref{x^2=By^p+Cz^p} with exponent $p$ of Type II has no asymptotic solution in $\mcO_K^3 \setminus S_r$.
	In particular, if $[K:\Q]$ is odd, then $S_1=\{(\pm1, -1,1)\}$ (resp., $\{(\pm1, 1,1)\}$) for $B=1$ (resp., $B=-1$), and $S_r= \phi $ for $r=2,4,5$. 
\end{prop}

\subsection{Steps to prove Theorem~\ref{main result1 for x^2=By^p+2^rz^p over K}.}
We now prove the modularity of the Frey curve $E:=E_{a,b,c}$ in \eqref{Frey curve for x^2=By^p+Cz^p of Type I} associated to any non-trivial primitive solution $(a,b,c)\in$ $\mcO_K^3 \setminus S_r$.
\begin{thm}
	\label{modularity of Frey curve of x^2=By^p+2^rz^p over K}
	Let $(a,b,c)\in \mcO_K^3 \setminus S_r$ be a non-trivial primitive solution to the equation~\eqref{x^2=By^p+Cz^p} with exponent $p$ of Type II, and let $E:=E_{a,b,c}$ be the associated Frey curve. Then, there exists a constant $A:=A_{K,B}$ (depending on $K,B$) such that for primes $p >A$, $E/K$ is modular.
\end{thm}

\begin{proof}
	Arguing as in the proof of Theorem~\ref{modularity of Frey curve x^2=By^p+Cz^p over W_K}, there exists $\lambda_k \in K$ with $1 \leq k \leq m$ such that $E/K$ is modular for all $\lambda(E) \notin\{\lambda_1, \lambda_2, ..., \lambda_m\}$.
	If $\lambda(E)= \lambda_k$ for some $k \in \{1, 2, \ldots, m \}$, then $\left(\frac{b}{c} \right)^p= \frac{-2^r\lambda_k}{B}$.
	The above equation determines $p$ uniquely, denoting it by $p_k$. Otherwise, we get $b=\pm c$.
	Since $a^2=Bb^p+2^rc^p$ and $(a,b,c)$ is primitive, we get $b=\pm1$ and $c=\pm1$, hence $(a,b,c) \in S_r$, which is not possible.
	Finally, the proof of the theorem follows by taking $A_K=\max \{p_1,...,p_m\}$.
\end{proof}

\subsubsection{Type of reduction with image of inertia.}
The following lemma specifies the type of reduction of the Frey curve $E:=E_{a,b,c}$ given in ~\eqref{Frey curve for x^2=By^p+Cz^p of Type I} at  $\mfP \in U_K$, when $(a,b,c)\in$ $\mcO_K^3$ and $r\in \{1,2,4,5\}$. More precisely;
\begin{lem}
	\label{reduction on T and S x^2=By^p+Cz^p Type II over K}
	Let $r\in \{1,2,4,5\}$. Let $(a,b,c)\in \mcO_K^3$ be a non-trivial primitive solution to the equation~\eqref{x^2=By^p+Cz^p} with exponent $p>(6+r)v_\mfP(2)$ of Type II, and let $E$ be the associated Frey curve. If $\mfP \in U_K$, then either $p | \#\bar{\rho}_{E,p}(I_\mfP)$ or $3 | \#\bar{\rho}_{E,p}(I_\mfP)$.
\end{lem} 

\begin{proof}
Recall that $\Delta_E=2^{r+6}B^2(b^2c)^{p}$ and $j_E=2^{6-r} \frac{(Bb^p+2^{r+2}c^p)^3}{B^2(b^2c)^{p}}$. If $\mfP |bc$, then $p | \#\bar{\rho}_{E,p}(I_\mfP)$ by Lemma~\ref{reduction on T and S} and due to the fact that $p>(6+r)v_\mfP(2)$. If $\mfP \nmid bc$, then $v_\mfP(j_E)= (6-r)v_\mfP(2)$ and $v_\mfP(\Delta_E)=(6+r)v_\mfP(2)$. Since $\mfP \in U_K$ and $r\in \{1,2,4,5\}$, we get $v_\mfP(j_E)\geq 0$ and $3 \nmid v_\mfP(\Delta_E)$.
Hence, by Lemma~\ref{3 divides discriminant}, we get  $3 | \#\bar{\rho}_{E,p}(I_\mfP)$.
%
\end{proof}

\subsection{Proof of Theorem~\ref{main result1 for x^2=By^p+2^rz^p over K}}	
The proof of this theorem depends on the following result.
\begin{thm}
\label{auxilary result x^2=By^p+2^rz^p over K}
Let $K$ be a totally real field satisfying $(ES)$, and let $r\in \{1,2,4,5\}$. 
Then, there is a constant $V:=V_{K,B}>0$ (depending on $K,B$) such that the following holds. Let $(a,b,c)\in \mcO_K^3 \setminus 
S_r$ be a non-trivial primitive solution to the equation~\eqref{x^2=By^p+Cz^p} with exponent $p >V$ of Type II, and let $E$ be the Frey curve as in ~\eqref{Frey curve for x^2=By^p+Cz^p of Type I}. Then there exists an elliptic curve $E^\prime/K$ such that:

\begin{enumerate}
	\item $E^\prime/K$ has good reduction away from $S_K^{\prime}$ and has a non-trivial $2$-torsion point, and
	$\bar{\rho}_{E,p} \sim\bar{\rho}_{E^\prime,p}$;
	\item For $\mfP \in U_K$, either $v_\mfP(j_{E^\prime})<0$ or $3\nmid v_\mfP(j_{E^\prime})$.
\end{enumerate}
\end{thm}
\begin{proof}
Arguing as in the proof of Theorem~\ref{auxilary result x^2=By^p+Cz^p over W_K}, the first part of Theorem~\ref{auxilary result x^2=By^p+2^rz^p over K} follows from \cite[Theorem 8]{FS15}, Theorem~\ref{conductor of E and E' x^2=By^p+Cz^p Type I}, Theorem~\ref{modularity of Frey curve of x^2=By^p+2^rz^p over K} and Lemma~\ref{reduction on T and S x^2=By^p+Cz^p Type II over K}. Let $\mfP \in U_K$. If $p | \# \bar{\rho}_{E,p}(I_\mfP)= \# \bar{\rho}_{E^\prime,p}(I_\mfP)$, then by Lemma~\ref{criteria for potentially multiplicative reduction}, we get $v_\mfP(j_{E^\prime})<0$. If $p \nmid \# \bar{\rho}_{E,p}(I_\mfP)$, then by Lemma~\ref{reduction on T and S x^2=By^p+Cz^p Type II over K}, we conclude that $3 | \# \bar{\rho}_{E,p}(I_\mfP)=\# \bar{\rho}_{E^\prime,p}(I_\mfP)$. If $v_\mfP(j_{E^\prime}) < 0$, then we are done. If $v_\mfP(j_{E^\prime}) \geq 0$, then  by Lemma~\ref{3 divides discriminant}, we have $3\nmid v_\mfP(\Delta_{E^\prime})$. Since $j_{E^\prime} = \frac{c_4^3}{\Delta_{E'}}$ and $3\nmid v_\mfP(\Delta_{E^\prime})$, we get $3\nmid v_\mfP(j_{E^\prime})$. This completes the proof of the theorem.
\end{proof}

\begin{proof}[Proof of Theorem~\ref{main result1 for x^2=By^p+2^rz^p over K}]
Let $r\in \{1,2,4,5\}$. Let $(a,b,c)\in \mcO_K^3 \setminus S_r$ be a non-trivial primitive solution to the equation~\eqref{x^2=By^p+Cz^p} with exponent $p >V$ of Type II, where $V:=V_{K,B}$ is the constant as in Theorem~\ref{auxilary result x^2=By^p+2^rz^p over K}. By Theorem~\ref{auxilary result x^2=By^p+2^rz^p over K}, there exists an elliptic curve $E^\prime/K$ having a non-trivial $2$-torsion point and good reduction away from $S_K^{\prime}$. 
By ~\eqref{assumption for main result1 x^2=By^p+2^rz^p over K}, there exists some $\mfP \in U_K$ that satisfies $|v_\mfP(\alpha\beta^{-1}) |\leq 6v_\mfP(2)$ and $v_\mfP(\alpha\beta^{-1})\equiv 0 \pmod 3$.

Now, arguing as in the proof of Theorem~\ref{main result1 for x^2=By^p+Cz^p Type II}, we find $v_\mfP(j_{E^\prime}) \geq 0$ by using $|v_\mfP(\alpha\beta^{-1}) |\leq 6v_\mfP(2)$.
Recall that $j_{E^\prime} = 2^8\frac{(1+\mu)^3}{\mu}$ with $\frac{\mu}{4}= \alpha\beta^{-1}$. This implies $v_\mfP(j_{E^\prime}) \equiv 2v_\mfP(2)-v_\mfP(\mu)= -v_\mfP(\alpha\beta^{-1}) \pmod3 $. 
Since $v_\mfP(\alpha\beta^{-1})\equiv 0 \pmod 3$, $v_\mfP(j_{E^\prime})\equiv 0 \pmod 3$ and hence $3 | v_\mfP(j_{E^\prime})$.
Therefore, $v_\mfP(j_{E^\prime}) \geq0$ and $3 | v_\mfP(j_{E^\prime})$, which contradicts Theorem~\ref{auxilary result x^2=By^p+2^rz^p over K}. This completes the proof of the theorem.
\end{proof}


\begin{proof}[Proof of Proposition~\ref{main result2 for x^2=By^p+2^rz^p over K}]
Since $B=\pm1$, and $2$ is inert in $K$, we find $S_K^\prime =S_K$ is principal. Let $\mfP \in U_K$ be the unique prime lying above $2$. Let $(\alpha, \beta, \gamma)\in  \mcO_{S_K}^\ast \times \mcO_{S_K}^\ast \times \mcO_{S_K}$ be a solution to $\alpha+\beta=\gamma^2$. According to Theorem~\ref{main result1 for x^2=By^p+2^rz^p over K}, it suffices to show that $|v_\mfP(\alpha\beta^{-1}) |\leq 6v_\mfP(2)$ and $v_\mfP(\alpha\beta^{-1}) \equiv 0 \pmod3$. If necessary by scaling even powers of $\mfP$ and swapping $\alpha, \beta$, we can assume $0\leq v_\mfP(\beta)\leq v_\mfP(\alpha)$ with $v_\mfP(\beta)=0$ or $1$.		
\begin{enumerate}
\item Suppose $v_\mfP(\beta)=1$. If $v_\mfP(\alpha)>1$, then $v_\mfP(\gamma^2)=v_\mfP(\alpha+\beta)=v_\mfP(\beta)=1$, which cannot occur because $v_\mfP(\gamma^2)$ is even. As a result, $v_\mfP(\alpha)=1$. Thus $|v_\mfP(\alpha\beta^{-1})|=0<6v_\mfP(2)$ and $v_\mfP(\alpha\beta^{-1}) \equiv 0 \pmod 3$.

\item Suppose $v_\mfP(\beta)=0$, i.e., $\beta$ is a unit in $K$.
\begin{itemize}
	\item Assume $\beta$ is a square. Then divide the equation $\alpha +\beta= \gamma^2$ by $\beta$ to obtain an equation of the form $\alpha'+1=\gamma'^2$, where $\alpha'= \alpha\beta^{-1} \in \mcO_{S_K}^\ast$, $\gamma' \in \mcO_{S_K}$. Using~\eqref{assumption for main result2 x^2=By^p+2^rz^p over K}, we get $|v_\mfP(\alpha\beta^{-1}) |\leq 6v_\mfP(2)$ and $v_\mfP(\alpha\beta^{-1}) \equiv 0 \pmod 3$.
	
	\item Assume $\beta$ is not a square. Consider the field $L=K(\sqrt{\beta})$. The minimal polynomial of $\beta$ is $m_\beta (x)=x^2-\beta$. Then $m_\beta (x)\in \mcO_K[x]$ with discriminant $4\beta$. Hence, $L$ is unramified away from $2$. By~\cite[Theorem 9(b)]{FKS20}, we conclude that $2$ totally ramified in $K$, which contradicts $2$ being inert in $K$.
\end{itemize}
\end{enumerate}
This completes the proof of the proposition.
\end{proof}

\section{Local criteria for the solutions of Diophantine equations}
\label{section for loc criteria}
In this section, we present several local criteria of $K$ which imply Theorems~\ref{main result1 for x^2=By^p+Cz^p Type I},~\ref{main result1 for  x^2=By^p+2^rz^p over K}.
We start this discussion with a lemma.

  \begin{lem}
  	\label{lemma to loc crit1 for x^2=By^p+2^rz^p over W_K}
 	\label{lem for loc crit1 for x^2=By^p+2^rz^p over K}
 	Suppose the $S_K^{\prime}$-unit equation $\lambda+\mu=1$ with $\lambda, \mu \in \mcO_{S_K^{\prime}}^\ast$ has only solutions $(-1,2), (2,-1)$ and $(\frac{1}{2}, \frac{1}{2})$. Then every solution $(\alpha, \gamma) \in \mcO_{S_K^{\prime}}^\ast \times \mcO_{S_K^{\prime}}$ to the equation $\alpha+1=\gamma^2$ satisfies $|v_\mfP(\alpha)|\leq 3v_\mfP(2)$ and $v_\mfP(\alpha) \equiv 0 \pmod 3$ for all $\mfP \in S_K$.
 \end{lem}
 \begin{proof} 
The solution $(\alpha, \gamma) \in \mcO_{S_K^{\prime}}^\ast \times \mcO_{S_K^{\prime}}$ to the equation $\alpha+1=\gamma^2$ gives rise to a solution of $\lambda +\mu=1$ with $\lambda, \mu \in \mcO_{S_K^{\prime}}$ as follows.
Take $\lambda :=\frac{\gamma+1}{2}, \mu :=\frac{1-\gamma}{2}$. Since $\gamma \in \mcO_{S_K^{\prime}}$, 
then so are $\lambda, \mu$. Since $\alpha =-4 \lambda \mu$ with $\alpha \in \mcO_{S_K^{\prime}}^\ast$, we get $\lambda, \mu \in \mcO_{S_K^{\prime}}^\ast$ with $\lambda +\mu=1$. The choices of $(\lambda, \mu) \in \{(-1,2), (2,-1), (\frac{1}{2}, \frac{1}{2}) \}$ implies $\alpha= -1$ or $8$. Therefore, $|v_\mfP(\alpha)|\leq 3v_\mfP(2)$ and $v_\mfP(\alpha) \equiv 0 \pmod 3$ for all $\mfP\in S_K$.
 \end{proof}

\subsection{Local criteria for Theorem~\ref{main result1 for x^2=By^p+Cz^p Type I}:}
\label{section for local criteria for Diophantine equations over $W_K$}
In this section, we give local criteria of $K$, which imply Theorem~\ref{main result1 for x^2=By^p+Cz^p Type I}. Throughout this section, we assume $B= \pm1$ and $C= \pm1$ or $2^r$ for some $r \in \N$ to get  $S_K^{\prime}=S_K$.
 
\begin{prop}[Quadratic]
	\label{loc crit for real quadratic field over W_K}
Let $K=\Q(\sqrt{d})$ for some prime $d$ with $d \equiv 5 \pmod8$. Then the conclusion of Theorem~\ref{main result1 for x^2=By^p+Cz^p Type I} holds over $K$.
\end{prop}

\begin{proof}
	Since $d \equiv 5 \pmod 8$, $K$ has discriminant $d$, $2$ is inert in $K$, and hence  
	$S_K$ is principal. The assumptions on $B,C$ give $S_K^{\prime}= S_K$.
		By~\cite[\S 3.8]{M22}, we have $2 \nmid h_K$. 
 	By~\cite[Table 1, \S 6]{FS15}, the $S_K$-unit equation $\lambda+\mu=1$ with $\lambda, \mu \in \mcO_{S_K}^\ast$ has only solutions $(-1,2), (2,-1)$ and $(\frac{1}{2}, \frac{1}{2})$. 
  Now the proof of the proposition follows from Lemma~\ref{lemma to loc crit1 for x^2=By^p+2^rz^p over W_K} and Proposition~\ref{main result2 for x^2=By^p+Cz^p Type I}.
\end{proof}
 
\begin{prop}[Odd degree]
	\label{loc crit for odd degree field over W_K}
	Let $K$ be a field of degree $n$ with $2 \nmid h_K$. Suppose 
	\begin{enumerate}
        \item there exists a prime $q \geq5$ with $\gcd(n, q-1)=1$ such that $q$ totally ramifies in $K$;
		\item $2$ is either inert or $2=\mfP ^n $ for some principal ideal $\mfP \in P$.
	\end{enumerate}
Then the conclusion of Theorem~\ref{main result1 for x^2=By^p+Cz^p Type I} holds over $K$.
\end{prop}

\begin{proof}
Let $\mfP \in S_K$ be the unique prime ideal lying above $2$. Arguing as in the proof of Lemma~\ref{lemma to loc crit1 for x^2=By^p+2^rz^p over W_K}, 
every solution $(\alpha, \gamma) \in \mcO_{S_K^{\prime}}^\ast \times \mcO_{S_K^{\prime}}$ to the equation $\alpha+1=\gamma^2$ gives rise to a solution $\lambda, \mu \in \mcO_{S_K^{\prime}}^\ast$ with  $\lambda +\mu=1$,
where $\lambda=\frac{\gamma+1}{2}$, $\mu=\frac{1-\gamma}{2}$, and a relation $\alpha=-4\lambda \mu$.
Since $S_K^{\prime}=S_K$, by~\cite[Lemma 4.1]{FKS21}, we get $ \max\{|v_\mfP(\lambda)|, |v_\mfP(\mu)|\} < 2v_\mfP(2)$. In particular, $|v_\mfP(\alpha)|= |2v_\mfP(2)+ v_\mfP(\lambda)+ v_\mfP(\mu)| < 6v_\mfP(2)$. Now, the proof of the proposition follows from Proposition~\ref{main result2 for x^2=By^p+Cz^p Type I}.
\end{proof}

\subsection{Local criteria for Theorem~\ref{main result1 for x^2=By^p+2^rz^p over K}:}
\label{section for local criteria for Diophantine equations over $K$}
In this section, we provide local criteria of $K$, which imply Theorem~\ref{main result1 for x^2=By^p+2^rz^p over K}. Throughout this section, we assume that $B=\pm1$ and $C=2^r$ with $r=1,2,4,5$ to get $S_K^{\prime}=S_K$.

\begin{prop}[Quadratic]
	\label{loc crit for real quadratic field over K}
Assume that Conjecture~\ref{ES conj} holds over $K=\Q(\sqrt{d})$ for some prime $d$ with $d \equiv 5 \pmod8$. 
Then the conclusion of Theorem~\ref{main result1 for x^2=By^p+2^rz^p over K} holds over $K$.
\end{prop}

\begin{proof}
 Since $d$ is a prime with $d\equiv 5 \pmod 8$, $2 \nmid h_K^+$, $2$ is inert in $K$ and hence $S_K^{\prime}=S_K$ is principal. Arguing as in the proof of Proposition~\ref{loc crit for real quadratic field over W_K}, 
 we find that the $S_K$-unit equation $\lambda+\mu=1$ with $\lambda, \mu \in \mcO_{S_K}^\ast$ has only solutions $(\frac{1}{2},\frac{1}{2} ), (2,-1)$ and $ (-1,2)$. Now the proof of the proposition follows from Lemma~\ref{lem for loc crit1 for x^2=By^p+2^rz^p over K} and Proposition~\ref{main result2 for x^2=By^p+2^rz^p over K}.	
\end{proof}
 
\begin{prop}[Odd degree]
	\label{loc crit2 for odd degree field over K}
Let $K$ be a field of degree $n>1$ with $2 \nmid h_K^+$. Suppose 
	\begin{enumerate}
 \item there exists a prime $q \geq5$ with $\gcd(n, q-1)=1$ such that $q$ totally ramifies in $K$;
		\item $2$ is inert in $K$.
\end{enumerate}
 Then the conclusion of Theorem~\ref{main result1 for x^2=By^p+2^rz^p over K} holds over $K$.
\end{prop}

Define $\Lambda_{S_K}:=\{(\lambda, \mu) : \lambda+\mu=1, \ \lambda,\mu \in \mcO_{S_K}^\ast\}$.
By the discussion in \cite[\S 5]{FS15},
the action of the symmetric group $\mfS_3$ on $\mbP^1(K) \setminus \{0,1, \infty\}$ induces an action on $\Lambda_{S_K}$ as $(\lambda, \mu)^\sigma :=(\lambda^\sigma, \mu^\sigma)$ with $\sigma \in \mfS_3$. For any $(\lambda, \mu) \in \Lambda_{S_K}$ and $\mfP \in S_K$, define $m_{\lambda, \mu}(\mfP) := \max \{|v_\mfP(\lambda)|, |v_\mfP(\mu)|\}$. If $\mfP \in S_K$ is unique, then we write $m_{\lambda, \mu}$ for $m_{\lambda, \mu}(\mfP)$. 
\begin{lem}
	\label{integral reduction of lambda, mu}
	Suppose $2$ is inert in $K$ and let $\mfP \in S_K$ be the unique prime lying over $2$. Then, for any $(\lambda, \mu) \in \Lambda_{S_K}$, there exists $(\lambda^\prime, \mu^\prime) \in \Lambda_{S_K}$ with $\lambda^\prime \in \mcO_K$, $\mu^\prime \in \mcO_K^\ast$ and $\sigma \in \mfS_3$ such that $(\lambda^\prime, \mu^\prime)= (\lambda, \mu)^\sigma$ and $m_{\lambda, \mu}= m_{\lambda^\prime, \mu^\prime}$.
\end{lem}

\begin{proof}
	If $v_\mfP(\lambda)= v_\mfP(\mu)=0$, or $v_\mfP(\lambda)>0$ (in this case $v_\mfP(\mu)=0$),
	then take  $\lambda^\prime=\lambda,\ \mu^\prime=\mu$ and $\sigma (\lambda)= \lambda$. 
	If $v_\mfP(\mu)>0$, then $v_\mfP(\lambda)=0$ and take $\lambda^\prime=\mu,\ \mu^\prime=\lambda$ and $\sigma (\lambda)=1-\lambda$. If $v_\mfP(\lambda)<0$ then $v_\mfP(\mu)= v_\mfP(\lambda)= -m_{\lambda, \mu}<0$ and take $\lambda^\prime= \frac{1}{\lambda}$, $\mu^\prime=1-\frac{1}{\lambda}$ and  $\sigma (\lambda)= \frac{1}{\lambda}$. In all cases, we can choose $\lambda^\prime \in \mcO_K$, $\mu^\prime \in \mcO_K^\ast$ with $m_{\lambda^\prime, \mu^\prime}=m_{\lambda, \mu}.$
\end{proof}
\begin{proof}[Proof of Proposition~\ref{loc crit2 for odd degree field over K}]
	Let $\mfP\in U_K$ be the unique prime lying above $2$.
By~\cite[Lemma 4.1]{FKS21}, we have $m_{\lambda, \mu} <2v_\mfP(2)=2$ for all $(\lambda, \mu) \in \Lambda_{S_K}$. By Lemma~\ref{integral reduction of lambda, mu}, there exists $\lambda^\prime \in \mcO_K$, $\mu^\prime \in \mcO_K^\ast$ such that
$m_{\lambda^\prime, \mu^\prime}<2$. If $m_{\lambda^\prime, \mu^\prime}=0$ then $\lambda^\prime, \mu^\prime \in \mcO_K^\ast$, contradicts~\cite[Theorem 4]{FKS21}. So, $m_{\lambda^\prime, \mu^\prime}=1$. Since $\lambda^\prime \in \mcO_K, \mu^\prime \in \mcO_K^\ast$, we get
$v_\mfP(\lambda^\prime)=1$ and hence $v_\mfP(\lambda^\prime \mu^\prime)=1=v_\mfP(2)$. By~\cite[Lemma 6.2(ii)]{FS15}, we have 
	\begin{equation}
		\label{valuation of lambda mu}
		v_\mfP(\lambda \mu)\equiv v_\mfP(2) \pmod3.
	\end{equation}
    We now show that the proof of Proposition~\ref{loc crit2 for odd degree field over K} follows from Proposition~\ref{cor to main result2 for x^2=By^p+2^rz^p over K}. Now, arguing as in the proof of Lemma~\ref{lemma to loc crit1 for x^2=By^p+2^rz^p over W_K}, 
    every solution $(\alpha, \gamma) \in \mcO_{S_K}^\ast \times \mcO_{S_K}$  to the equation $\alpha+1=\gamma^2$ 
gives rise to an element $(\lambda,\mu)\in \Lambda_{S_K}$ with a relation $\alpha=-4\lambda \mu$.
Since $m_{\lambda, \mu} <2$, $|v_\mfP(\alpha)| \leq 2 + |v_\mfP(\lambda )|+ |v_\mfP(\mu)|<6=6v_\mfP(2)$. By \eqref{valuation of lambda mu}, we get $ v_\mfP(\alpha)\equiv 0 \pmod3$. Note that, here $n$ is odd, hence $K$ satisfies $(ES)$. 
\end{proof}

\section{Acknowledgments}  
     The authors are thankful to the anonymous referee for the valuable suggestions to the improvement of this paper. 
     This research was supported in part by the International Centre for Theoretical Sciences (ICTS) for participating in the program - ICTS Rational Points on Modular Curves (code: ICTS/RPMC-2023/09).

\end{document}